\documentclass[12pt, reqno]{amsart}
\usepackage{amsmath,amssymb,amsthm,mathrsfs}
\usepackage{cleveref}
\usepackage{enumerate}
\usepackage{color}
\usepackage{fullpage}
\usepackage{mathptmx}
\usepackage[all,cmtip]{xy}
\usepackage{url}

\newtheorem{theorem}{Theorem}
\newtheorem*{introthm*}{Main Theorem} 
\newtheorem{introthm}[theorem]{Theorem} 
\newtheorem{thm}{Theorem}[section]
\newtheorem{prop}[thm]{Proposition}
\newtheorem{proposition}[thm]{Proposition}
\newtheorem{lemma}[thm]{Lemma}
\newtheorem{cor}[thm]{Corollary}

\theoremstyle{definition}

\newtheorem{definition}[thm]{Definition}
\newtheorem{remark}[thm]{Remark}
\newtheorem{example}[thm]{Example}

\newtheorem{remarkQuestion}[thm]{Remark and Question}

\newcommand{\CC}{\mathbf{C}}
\newcommand{\m}{\mathfrak{m}}
\newcommand{\OO}{\mathscr{O}}

\newcommand{\tA}{\tilde{A}}

\newcommand{\cO}{\mathcal{O}}
\newcommand{\II}{\mathscr{I}}

\DeclareMathOperator{\lct}{lct}
\DeclareMathOperator{\Spec}{Spec}
\DeclareMathOperator{\Bl}{Bl}

\DeclareMathOperator{\codim}{codim}

\DeclareMathOperator{\Hom}{Hom}
\DeclareMathOperator{\ord}{ord}

\DeclareMathOperator{\Exc}{Exc}

\title{Log canonical thresholds of generic links of generic determinantal varieties}

\author{Youngsu Kim}
\address{Department of Mathematics, California State University San Bernardino, San Bernardino, CA 92407}
\email{youngsu.kim@csusb.edu}

\author{Lance Edward Miller}
\address{Mathematical Sciences, University of Arkansas, Fayetteville, AR 72701}
\email{lem016@uark.edu}

\author{Wenbo Niu}
\address{Mathematical Sciences, University of Arkansas, Fayetteville, AR 72701}
\email{wenboniu@uark.edu}

 %; Secondary {} } 
\begin{document}

\maketitle

\date{\today}
\begin{abstract}
The article concerns the behavior of determinantal varieties under generic linkage. In particular, it was shown that one has a general inequality \cite{Niu14} of log canonical thresholds on passing to generic linkage. It is immediate to verify this can be strict. 
Except a few special classes, it is not known under which conditions force the equality or strict inequality to occur.  
We demonstrate determinental varieties constitute a class of varieties for which equality is attained.
\end{abstract}

% {\color{blue}\noindent Note to Lance and Wenbo. \\
% The extension $B$ is now named as $A'$ to be consistent with our use of $-'$.\\
% A, A' schemes \\
% R, S rings\\
% T,T' sets of indeterminates\\
% $\lct (X) \rightarrow \lct(A,X)$\\
% $\codim (X) \rightarrow \codim_A X$\\
% $\ord_E (X)$ has inputs of schemes not ideals (this was done to uniformize our notation)\\
% non-singular $\rightarrow$ smooth (quasi-projective)}

\section{Introduction}
The log canonical threshold of a complex algebraic variety $X$ is a numerical invariant measuring singularities of $X$. 
Let $A$ be an affine or projctive space over $\CC$. 
Suppose that $X$ is a reduced equidimensional subscheme of $A$.  
Consider the generic link $Y$ of $X$, which is defined in an affine extension $A'$ of $A$, 
via a generic complete intersection $V$ of the equations defining $X$, see \cref{secDefNota} for the precise definition. 
The work of the third author \cite[Prop.\ 3.7]{Niu14} gives the following relationship on the log canonical thresholds of $X$, $Y$, and $V$, 
\begin{equation}\label{eqEq}
\lct(A',Y) \geq \lct(A',V) = \lct(A, X).
\end{equation}

In general, inequality in \Cref{eqEq} can be strict. For instance, the log canonical threshold of the generic link of a hypersurface is always $1$, whereas there are hypersurfaces whose log canonical threshold is strictly less than $1$. 
In \Cref{xmp:codim2}, we provide examples demonstrating a wild behavior of the log canonical thresholds of generic links in the codimension $2$ case. 
%???It is worthwhile to keep in mind that even calculating the log canonical threshold of hypersurfaces is a deep and interesting subject on its own right \cite{GHM16, CPW14, BGG12}.
%{\color{red}??? Referee item (2)}

To the best of our knowledge, there is no known non-trivial classe of varieties where equality in \Cref{eqEq} is proved to hold. 
Determinantal varieties and ideals are classical objects in algebraic geometry and commutative algebra. 
The goal of this article is to prove that in \Cref{eqEq} one has equality throughout  when $X$ is a generic determinantal variety.

\begin{introthm}\label{mainthm1}
Let $X$ be a generic determinantal variety and $Y$ its generic link. 
Then $X$ and $Y$ have the same log canonical threshold. 
\end{introthm}

Our approach to prove \Cref{mainthm1} is to utilize a log resolution of the pair $(A,X)$ \cite{Vai84}. 
A log resolution of the pair $(A,X)$ provides a quick and easy way to determine the log canonical threshold of $X$ in $A$.  
Such a log resolution is explicitly described in \cite{Joh03} and \cite{Doc13}. 
In general, a log resolution of $(A,X)$ does not extend to a log resolution of $(A',Y)$. 
Thanks to \Cref{eqEq}, to prove our theorem it suffices to compare the orders of $X$ and $Y$ along the exceptional divisors associated to the log resolution of $X$. 
Indeed, we can relate these numbers explicitly by utilizing facts about generic determinantal varieties and linkage theory, see \Cref{genDegreeLem}. 

\bigskip

{\bf Acknowledgments:} The authors would like to thank Roi Docampo, William Heinzer, William D. Taylor, and Bernd Ulrich for valuable discussions. We also thank the referee for very detailed readings and questions which improved the exposition and mathematics of the work significantly.

\section{Preliminaries}\label{secDefNota}
%Notation and overview of the proof}

Throughout, all schemes are quasi-projective over $\CC$, and all ring are commutative with 1. 
For a closed subscheme $Z$ of a scheme $A$, we denote by $\mathcal{I}_{Z/A}$ the ideal sheaf of $Z$ in $A$. 
When the ambient space of $Z$ is clear, we also use the notation $\mathcal{I}_Z$. 
Let $R$ be a ring and $M = ( x_{ij}) $ an $m$ by $n$ matrix of indeterminates. 
We use the notation $R[M]$ to denote the polynomial ring $R[\underline{x_{ij}}] = R[ \, x_{ij} \mid 1 \le i \le m, 1 \le j \le n]$.
In $R[M]$, $I_i (M)$ denotes the ideal generated by $i$ by $i$ minors of $M$, and $I_i (M)$ is $R$ if $i = 0$ and $0$ if $i > \min\{m,n \}$.

%varieties are assumed to be reduced and irreducible schemes 
\subsection{Linkage and generic linkage}\label{sec:linkage}
Let $A$ be a smooth quasi-projective variety and $X, Y$ equidimensional subschemes in $A$. 

\begin{definition}[{\cite[Prop.\ 1.1]{PS74},\cite[Def.\ 2.1]{HU85}}]\label{def:linkage}
We say that $X$ and $Y$ are {\it linked} via $V$ if $V = X \cup Y$ is a complete intersection such that 
\begin{equation*}
\II_{Y}/\II_{V} \cong \Hom_{\OO_A} ( \OO_X, \OO_V)~\hbox{and}~\II_{X}/\II_{V} \cong \Hom_{\OO_A} ( \OO_Y, \OO_V).
\end{equation*}
In addition, if $X$ and $Y$ do not have any common component, we say $X$ and $Y$ are {\it geometrically linked}. 
\end{definition}
Note that if $X$ and $Y$ are linked via $V$, then $\codim_A X = \codim_{A} Y = \codim_{A} V$. 
Now, we introduce the notion of generic linkage following \cite[Def.\ 2.3]{HU85}.
Let $R$ be a regular ring, $I_X$ an equidimensional ideal of codimension $c$, $A = \Spec R$, and $X = \Spec R/I_X$.
Fix a generating sequence $g_1,\dots, g_\mu$ of $I_X$. %, and let $G$ denote the column vector with entries $g_i$. 
Form a $c \times \mu$ matrix $T$ of indeterminates. 
In the ring $S = R[T]$, for $1 \le j \le c$, define
\begin{equation*}
f_j := t_{j1} g_1 + \cdots + t_{j\mu} g_{\mu},
\end{equation*}
and set $I_V = (f_1,\dots, f_c)$. 
The ideal $I_V$ is a complete intersection ideal of codimension $c$ in $S$, see \cite{Hoc73}.
The {\it generic link} $Y$ of $X$ is the subscheme of $A' := \Spec S$ defined by $I_Y := I_V :_S I_X S = \{ f \in S \mid f I_XS \subset I_V \}$. We also say, the ideal $I_Y$ is the {\it generic link} of $I_X$ via $I_V$.
A generic link is irreducible and reduced \cite[Thm.\ 2.4 or 2.5]{HU87}. 

\begin{remark}\label{rmkIndepGenLink}
%%\begin{enumerate}[$(a)$]
%%\item 
A generic link of $X$ depends on the choice of a generating set for $I_X$. 
We use the notion of equivalence of generic links in \cite[p. 1270]{HU85}.
Let $(R,I)$ denote a pair of a ring and an ideal. 
We say $(R,I)$ and $(R',I')$ are {\it equivalent} if for some sets of indeterminates $T$ and $T'$, 
there exists a ring isomorphism $\phi \colon R[T] \to R'[T']$ such that $\phi(IR[T]) = I' R'[T']$.
By \cite[Prop. 2.4]{HU85}, if $Y$ are $Y'$ are generic links of $X$, then $I_Y$ and $I_{Y'}$ satisfy the equivalence defined above. 
In \Cref{thm:indendenceOfOrder}, we show that this does not create an issue for our computation of the log canonical threshold. 
%In our notation, the isomorphism constructed for the equivalence of generic links fixes $\CC[M]$. 
%So, if $U$ and $U'$ are sets of variables for generic links of $Y$ and $Y'$ of $I_X$, respectively, then one has that $\ord_{I \CC[M,U]} I_Y = \ord_{I \CC[M,U']} I_{Y'}$ for any $\CC[M]$-ideal $I$.
This is why sometime we refer to $Y$ as {\it the} generic link.
%%
%%\item ~
%%\end{enumerate}
\end{remark}

\subsection{Log canonical threshold}
The log canonical threshold is a measure of singularities of pairs. 
It is a rational number which can be computed from a log resolution. 
We refer the reader to \cite{Kol97, Lar04} for the details of this invariant. 
Let $X \subset A$ be a subvariety of a smooth quasi-projective variety $A$. 

\begin{definition}[{\cite[Thm.\ 0.2]{KM98}}]
A {\it log resolution} of a pair $(A,X)$ is a projective birational morphism $\pi \colon \tA \to A$ such that 
$\tA$ is smooth, $I_X \cO_{\tA}$ is invertible, and $\Exc(\pi) \cup \pi^{-1}(X)$ has simple normal crossing support. 
Here, $\Exc(\pi)$ denotes the set of exceptional divisors of $\pi$. 
\end{definition}

%By Hironaka's result \cite{Hir}, such a embedded resolution exists for the pair $(A,X)$. 
%In this note, we only consider embedded log resolutions, so we will drop the adjective `embedded'. 
If $\pi \colon \tA \to A$ is a log resolution of a pair $(A,X)$ and $\operatorname{Supp} ( \Exc(\pi) \cup \pi^{-1}(X) ) = \{ E_1, \dots, E_\ell \}$, 
then we may write $I_X \cO_{\tA} = \cO_{\tA}(-G)$ for some divisor $G = \sum a_{i,X} E_i$ and $K_{\tA/A} \allowbreak {:= K_{\tA} - \pi^* K_A} \allowbreak {= \sum k_i E_i}$, where $K_{\tA/A}$ denotes the relative canonical divisor. Such a log resolution provides the following characterization of the log canonical threshold  
\begin{equation}\label{def:lct}
\lct(A,X) = \min_i \left\{ \frac{k_i + 1}{ a_{i,X} }\right\}.
\end{equation}
We use this as our definition of a log canonical threshold. 
%When the ambient space is clear, we write simply $\lct(X)$ for $\lct(A,X)$.

As stated in the introduction, for $X$ defined in affine space $A$ with generic link $Y$ in an affine extension $A'$, one has the general inequality \cite[Prop. 3.7]{Niu14}
\begin{equation*}
\lct(A',Y) \geq \lct(A, X).
\end{equation*}

First, we present a few examples regarding this inequality.
Understanding when equality is attained is our primary goal in this note.  
In the introduction, we mentioned that the inequality can be strict for the hypersurface case. 
Our examples below demonstrate that even in the codimension two case both cases can happen and determining when one has $\lct(A',Y) = \lct(A, X)$ is not a simple task.
The underlying difficulty lies in determining $\lct(A',Y)$ as the equations defining $Y$ becomes significantly more complicated than those defining $X$. 
 
%Fix $R = \CC[x,y,z,a_1,a_2,b_1,b_2]$ viewed as the coordinate ring of an affine space $A$. The subscheme $X$ define by $(x^2z,y^3)$ is a complete intersection and has generic link $Y$ define by $( x^2z a_1 + y^3 a_2, x^2z b_1 + y^3 b_2, a_2b_1 - a_1b_2)$. All ideals in question are binomial, and so one may utilize the linear programming methods in \cite{ST09}, one may verify that $\lct(A,X) = 5/6$ where as $\lct(A,Y) = 11/6 > 5/6$.

%In contrast, for a generic complete intersection $X$ of linear forms, up to equivalence, one may assume this is defined by $(x,y)$. It is easy to show that $\lct(A,X) = 2$ and its generic link $Y$ given by $(a_1 x + b_1 y, a_2 x + b_2 y, a_1 b_2 - a_2 b_1)$ also has $\lct(A,Y) = 2$. 

%{\color{red} This is just for our internal copy. The paper phrases a linear programming problem, which is more reliable than M2. I computed $\lct(A,Y)$ in two ways, first using the Dmodules package in M2. I got the same answer as the associated linear programming problem from the Shibuta-Takagi paper.  That problem has $6$ variables, $10$ constraints, and is maximized at $(1/2,0,1,0,1/3,0)$ for the function which sums the coordinates, and indeed $1/2 + 1 + 1/3 = 1 + 5/6 = 11/6$. I don't think its instructive to include this matrix/linear programming. }

\begin{example} \label{xmp:codim2} We deal with codimension two subschemes in an affine space $A$. 
For any subscheme $X$ of $A$, $Y$ denotes the generic link of $X$ in $A'$ following the definition stated after \Cref{def:linkage}. 
In both cases, we use the fact that $$\lct(A,X) \le \lct (A',Y) \le \codim_A Y = \codim_A X = 2,$$ see \cite[Lem. 2.5]{Niu14} or \cite[Ex. 9.2.14 and Prop. 9.2.32(a)]{Lar04} for the second inequality.
In particular, if $\lct(A,X) = 2$, then $\lct(A,X) = \lct(A',Y) = 2$. 
\begin{enumerate}[$(a)$]
\item The case of complete intersections: Suppose $R = \CC[x_1,x_2,x_3]$, $S = R[t_{ij} \mid 1 \le i \le 2, 1 \le j \le 2]$, $A = \Spec R$, and $A' = \Spec S$. 
The subscheme $X$ defined by $x_1^2x_2,x_3^3$ is a complete intersection, and $Y$ is defined by $x_1^2x_2 t_{11} + x_2^3 t_{12}, x_1^2x_2 t_{21} + x_3^3 t_{22}, t_{12}t_{21} - t_{11}t_{22}$ in $A'$. 
All equations in question are binomial.
By the linear programming methods in \cite{ST09}, one verifies that $\lct(A,Y) = 11/6 > 5/6 = \lct(A,X)$.

In contrast, for a generic complete intersection $X$, for instance, $X = V(x_1,x_2)$, one has $\lct(A,X) = 2$.
Thus, $\lct (A,X) = \lct(A',Y)  = 2$. % given by $(t_{11} x_1 + t_{21} x_2, t_{12} x_1 + t_{22} x_2, _{12}t_{21} - t_{11}t_{22})$ also has $\lct(A,Y) = 2$. 

\item The case of non-complete intersections: % \label{xmp:nonCI} <- not used anywhere
Suppose $R = \CC[x_1,x_2,x_3]$, $S = R[t_{ij} \mid 1 \le i \le 2, 1 \le j \le 3]$, $A = \Spec R$, and $A' = \Spec S$. 
Further, suppose that $X$ is defined by $x_1x_2,x_2x_3,x_3x_1$. 
Since these equations are monomials, by \cite{How01}, $\lct(A,X) = 2$. 
Thus, $\lct(A,X) =  \lct(A',Y) =2$. 

However, when $X$ is defined by $x_1^2,x_1x_2,x_2^2$, it is easy to see that $\lct(A,X) = 1$. Notice that $(x_1^2,x_1x_2,x_2^2) = (x_1,x_2)^2$ and use the result in item $(a)$. 
A Macaulay2 computation shows the log canonical threshold of the generic link $Y$ of $X$ is $2$ thus $\lct(A',Y) = 2 > \lct(A,X) = 1$. 
%In this case, we do not know any method to compute the log canonical threshold of $Y$ even though there is a strong computational and theoretical indication that it is $2$. 
\end{enumerate}
\end{example}
%%
%%This is a determinantal variety defined by the 2 by 2 minors of the matrix 
%%\begin{equation*}
%%\begin{pmatrix}
%%x_1 & x_2 & 0 \\
%%0 & x_1 & x_2
%%\end{pmatrix}.
%%\end{equation*}
%%Its generic link is also a determinantal variety defined by the 3 by 3 minors of the matrix
%%\begin{equation*}
%%\begin{pmatrix}
%%x_1 & x_2 & 0 \\
%%0 & x_1 & x_2 \\
%%t_{11} & t_{12} & t_{13} \\
%%t_{21} & t_{22} & t_{23} 
%%\end{pmatrix},
%%\end{equation*}

%%\begin{example}[Non-complete intersections] \label{xmp:nonCI}
%%If $X$ is defined by $(xy,xz,yz)$ and $Y$ its generic link. One may verify that $\lct(A,X) = 2 = \lct(A,Y)$. {\color{red} Perhaps the easiest way to do this is to pass to the general link, that is specializing the variables $t_{ij}$ to a Zariski open subset. There are natural inequalities between the log canonical thresholds of the generic versus general link by such a restriction, but the general link itself is binomial which again using \cite{ST09}. \\
%%\color{blue} the codimension of $(xy,xz,yz)$ is two which is the maximum for the lct of $(A,Y)$. 
%%}
%%
%%In contrast, for $X$ defined by $(x^2y,x^3z)$, one has $\lct(A,X) = 1/2$ where as for $Y$ the generic link, $\lct(A,Y) = 2$.
%%
%%\end{example}

\subsection{Log resolution of a generic determinantal variety}\label{sec:setting} 
Let $R = \CC[M]$, where $M = ( x_{ij}) $ an $m$ by $n$ matrix of indeterminates. Write $A = \Spec R$ and $X = \Spec R/I_r(M)$ for some positive integer $r$. 
Here, $X$ is the variety parameterizing $m$ by $n$ matrices of rank at most $r-1$ with entries in $\CC$.
We assume $m \geq n \geq r$ and set $c = (m-r+1)(n-r+1)$ the codimension of $X$ in $A$. 

\

The log resolution of $(A,X)$ we refer in the paper is described in detail in \cite[Thm. 4.4 and Cor. 4.5]{Joh03}. 
For the convenience of the reader, 
we summarize these results in \Cref{thm:res} and describe the key step in \Cref{rmk:Joh03}. 
%For an $m$ by $n$ matrix $M$ with entries in a ring $R$.  

\begin{thm}\label{thm:res} 
Let $M$ be an $m$ by $n$ matrix of indeterminates, $A = \Spec \CC[M]$, and $X = \Spec \CC[M]/I_r(M)$. Set $A_0 := A$.
For $i = 1,\ldots, r$, let 
\begin{equation*}
\pi_i : A_i := \Bl_{Z_{i-1}} (A_{i-1}) \to A_{i-1}, 
%\pi_i : A_i := \Bl_{ {V(I_i(M))}^\sim } (A_{i-1}) \to A_{i-1}, 
\end{equation*}
where $Z_{i-1}$ denotes the strict transform of $V(I_i(M))$ in $A_{i-1}$ under the map $\pi_1 \circ \cdots \circ \pi_{i-1}$. 
By setting $\pi :=  \pi_1 \circ \cdots \circ \pi_r$ and  $\tA := A_r$, 
$\pi  \colon \tA  \to A$ 
 is a log resolution of the pair $(A,X)$. 
Let $\operatorname{Supp} ( \Exc(\pi) \cup \pi^{-1}(X) ) = \{ E_1, \dots, E_r \}$, where $E_i$ denotes strict transform of $V(I_i(M))$ in $\tilde{A}$. %the exceptional divisor introduced by $\pi_i$.
One has
\begin{equation*}
I_X \cO_{\tA} = \cO_{\tA}( - \sum a_{i,X} E_i )~\hbox{and}~K_{\tA/A} = \cO_{\tA} (- \sum k_i E_i)
\end{equation*}
where $a_{i,X}= r - i + 1$ and $k_i = (m-i+1)(n-i+1)-1$ for $i = 1, \dots, r$.
In particular, one has 
\begin{equation}\label{detLCTformula}
\lct(A,X) = \min_{0 \leq t \leq r-1} \frac{(m-t)(n-t)}{r - t}.
\end{equation}
\end{thm}

We note that $a_{i,X} = \ord_{E_i} (X)$.

\begin{remark}\label{rmk:Joh03}
The key point of the log resolution in \Cref{thm:res} is that each blow-up $\pi_i$ can be covered by an affine chart, 
where each chart looks ``essentially'' the same, 
and that each blow-up can be described iteratively. 
We describe the blow-up maps as in \cite[Sec.\ 4]{Joh03}.
%Write $M = ( x_{ij} \mid 1 \le i \le m, 1 \le j \le n )$ for the matrix of indeterminates. 
The first blow-up $\pi_1 \colon A_1= \Bl_{V(I_1(M) )} \to A_0$ is covered by the $mn$ affine charts corresponding to the indeterminates $x_{ij}$. 
We denote them by $U_{ab} \cong \Spec \CC[ y_{ij} \mid 1 \le i \le m, 1 \le j \le n]$, where $y_{ij} = {x_{ij}}/{x_{ab}}$ if $(i, j) \neq (a,b)$ and $y_{ab}  = x_{ab}$.
%We describe $U_{11}$.
On the affine chart $U_{11}$, the restriction $\pi_1 |_{U_{11}} \colon U_{11} \to A_0$ of $\pi_1$ corresponds to the ring homomorphism
\begin{equation*}
\CC[ x_{ij} \mid 1 \le i \le m, 1 \le j \le n] \to  \CC[ y_{ij} \mid 1 \le i \le m, 1 \le j \le n],
\end{equation*}
where $x_{ij}$ maps to $x_{11} = y_{11}$ if $(i,j) = (1,1)$ and to $y_{11} \cdot y_{ij}$ otherwise.
% Notation \underline{y_ij}
Note that $y_{11}$ is the local equation of the exceptional divisor $E_1$ on $U_{11}$, 
and in the coordinate ring $\CC[ \underline{y_{ij}} ]$ of $U_{11}$, the matrix $M$ can be written as follows
\begin{equation*}
M
 = \begin{pmatrix}
x_{11} & x_{12} & \cdots & x_{1n} \\
x_{21} & x_{22} & \cdots & x_{2n} \\
\vdots &  & \ddots & \vdots \\
x_{m1} & x_{m2} & \cdots & x_{mn}
\end{pmatrix}
= y_{11} 
\begin{pmatrix}
1 & y_{12}  & \cdots & y_{1n}  \\
y_{21}  & y_{22}  & \cdots & y_{2n}  \\
\vdots &  & \ddots & \vdots \\
y_{m1}  & y_{m2}  & \cdots & y_{mn} 
\end{pmatrix}.
\end{equation*}
We set $f_{ij} = y_{ij} - y_{i1} y_{1j}$ for $(i,j) \neq (1,1)$ and set
\begin{equation*}
M_1 :=  \left(\begin{array}{c|ccc}
1 & 0  & \cdots &0 \\
\hline
0  & {f_{22}}  & \cdots & f_{2n}  \\
\vdots &  & \ddots & \vdots \\
0  & f_{m2}  & \cdots & f_{mn}
\end{array} \right)
= \left( \begin{array}{cc} 
1 & 0 \\
0 & M_1' 
\end{array} \right),
\end{equation*} 
where $M_1'$ is the matrix obtained by deleting the first row and column of $M_1$. 
By \cite[Rmk.\ 4.3]{Joh03}, the set $\{ f_{ij} \mid (i,j) \ge (2,2) \} \cup \{ y_{11}, y_{1j}, y_{i1} \mid i,j > 1 \}$ is another coordinate system for the coordinate ring $\CC[ \underline{y_{ij}} ]$ of $U_{11}$. 
Under this new coornidates, $M_1$ is an $(m-1)$ by $(n-1)$ matrix of indeterminates.
Observe that for any $\ell$, $I_\ell(M) = y_{11}^\ell I_\ell(M_1) = y_{11}^\ell I_{\ell-1}(M_1')$.
Thus, 
%$I_2(M) = y_{11}^2 I_2(M_1) = y_{11}^2 I_1 (M_1')$, and 
the strict transform of $I_2(M)$ on $U_{11}$ is $V(I_2(M))^\sim|_{U_{11}} = V(I_1(M_1'))$, and this the center of the second blow-up $\pi_2$ on $U_{11}$.  
Similarly,
the strict transform of $X = V( I_r(M))$ on $U_{11}$ is $V(I_r(M))^\sim|_{U_{11}} = V(I_{r-1}(M_1'))$.
It is straightforward to see that this description does not depend on the the affine chart $U_{11}$ up to the change of the coordinates.
 
The further blow-ups are constructed iteratively on each affine charts $U_{ij}$. 
Each blow-up $\pi_i$ has the effect of dropping each index $m,n,$ and $r$ exactly by $1$, respectively. 
Thus, after $r$ successive blow-ups we reach the stated resolution in \Cref{thm:res}. 
We note that such a resolution $\pi$ in \Cref{thm:res} is a log resolution by \cite[Thm.\ 4.4]{Joh03}.
Furthermore, this log resolution is a factorizing resolution of the pair $(A,X)$ see \cite[Def. 2.6]{Niu14}.
%The singular locus of $X$ is precisely $V(I_{r-1}(M))$. 
% Johnson's thesis comes here
\end{remark}

%%\begin{example}
%%Suppose $X$ is defined by the $2$-minors of the matrix 
%%$
%%M = \begin{pmatrix} x_{11} & x_{12} & x_{13} \\
%%x_{21} & x_{22} & x_{23} 
%%\end{pmatrix}.
%%$  The log resolution described above will take two steps, blowing up successively the singular locus. In this case, this is given as the origin of the ambient affine space. On the $U_{11}$-chart of the first blowup, $x_{11}$ is a unit, so setting $y_{ij} = x_{ij}/x_{11}$ for $(i,j) \neq (1,1)$ and $y_{11} := x_{11}$, after row/column operations the coordinate matrix becomes $	
%%M_1 = \begin{pmatrix} 1 & 0 & 0 \\
%%0 & f_{22} & f_{23} 
%%\end{pmatrix}
%%$ where $f_{22} = y_{22} - y_{21}y_{12}$ and $f_{23} = y_{23} - y_{21}y_{13}$. The exceptional divisor on this chart is defined by $y_{11}$. The center of the next blow up is defined by $(f_{22},f_{23})$ and this last blow up completes the log resolution.	
%%\end{example}

\subsection{Order and degree}
%\begin{enumerate}
%\item
Let $E$ be a prime divisor over a smooth quasi-projective variety $A$, and let $c_A(E)$ denote the center of $E$ on $A$. 
Further, let $X$ be a closed subscheme $A$. 
% and let $\mathcal{I}_X$ and $\mathcal{I}_{c(E)}$ denote the ideal sheaves of $X$ and $c(E)$, respectively. 
The {\it order} of $X$ along $E$ is $\ord_E X := \sup \{ n \ge 0 \mid \mathcal{I}_X \subset (\mathcal{I}_{c_A(E)})^n \}$ \cite[Def.\ 9.36]{Lar04}. 
\
%This number is finite if $X$ contains $c_A(E)$. <= it should say the number is positive. I dropped this as we don't need this. 

%\item Assume the setting of \cref{thm:res}. 
%\end{enumerate}

Now, we prove the claim we made in \Cref{rmkIndepGenLink}.

\begin{proposition}\label{thm:indendenceOfOrder}
Let $X$ be an equidimensional subscheme of an affine space $A$, and $Y_1, Y_2$ be generic links of $X$ in $A'$ and $A''$, respectively. 
Further, let $E$ be a divisor over $A$, % such that $c_A(E) \subset X$
 and let $E'$ and $E''$ denote their pull-backs to $A'$ and $A''$, respectively. One has
\begin{equation*}
\ord_{E'} (Y') = \ord_{E''} (Y'').
\end{equation*}
In particular, the order of the generic link of $Y$ with respect to the extension of $E$ is independent of the choice of a generic link of $X$. % if $c_A(E) \subset X$. 
\end{proposition}
\begin{proof}
The proof follows from \Cref{lem:algOrd} which is the algebraic counterpart.
\end{proof}

For ideals $I, J$ in a ring $R$, we set $\ord_I J := \sup \{ n \ge 0 \mid J \subset I^n \}$. 

\begin{lemma}[{cf. \cite[The proof of Prop.\ 2.4]{HU85}}]\label{lem:algOrd}
Let $R$ be a regular ring and $I_X$ an equidimensional ideal of $R$. 
%Set $A = \Spec R$ and $X = \Spec R/I_X$. 
Let $I_{Y}$ and $I_{Y'}$ be generic links of $I_X$ in $S$ and $S'$, respectively.
For any $R$-ideal $I$, one has that 
\begin{equation*}
\ord_{I S} I_Y = \ord_{I S'} I_{Y'}.
\end{equation*}
\end{lemma}

\begin{proof}
Let $G,G'$, $T,T',\underline{f_i}$, and $\underline{f_i'}$ denote the generating sets, sets of determinates, and generic complete intersections to construct the generic links of $I_Y$ and $I_{Y'}$ of $I_X$, respectively, see \Cref{sec:linkage}.
Since one can compare $G$ and $G'$ to their union $G \cup G'$, 
by induction, it suffices to show the case $G = \{ f_1,\dots, f_\ell \}$ and $G' = G \cup \{ g \}$. 
Thus $S = R[T]$ and $S' = R[T']$. 
Let $c = \codim I_X$. 
We write that $T = ( t_{ij} )$, where $1 \le i \le c, 1 \le j \le \ell$ and that $T' = ( t'_{ij} )$, where $1 \le i \le c, 1 \le j \le \ell+1$.
Since $G$ is a generating set of $I_X$ and $g \in I_X$, 
one may write $g = \sum^\ell a_i f_i$, where $a_i \in R$. 
By the proof of \cite[Prop.\ 2.4]{HU85}, the ring homomorphism
\begin{equation*}
\phi \colon S[t_{1,\ell+1},\dots,t_{c,\ell+1}] \to S',
\end{equation*}
where 
\begin{align*}
\phi(t_{i, \ell+1}) &= t'_{i, \ell+1} \\
\phi(t_{i,j}) &= t'_{i,j} + a_j t_{i,\ell+1},
\end{align*}
is an isomorphism providing the equivalence of $(S, I_Y)$ and $(S', I_{Y'})$. 
In particular, $\phi$ fixes $R$ and $\phi(I_Y) = I_Y'$. 
Since $\phi$ fixes $R$, for any $R$-ideal $I$, 
$I_Y \subset I^p$ if and only if $I_Y' \subset I^p$ for any $p \ge 0$. 
This completes the proof.
\end{proof}

Lastly, we will use the following results in linkage theory to prove \Cref{thmAxAy}. 

\begin{proposition}\label{genDegreeLem} Let $S = R[M]$, where $R$ is a commutative noetherian ring containing an infinite field, 
$M$ is an $m \times n$ matrix of intermediates with $m \ge n$, 
and $I_X = I_r(M)$. % is the ideal of $S$ generated by $r \times r$ minors of $M$. 
The ring $S$ is a standard graded ring generated by the entries of $M$ over $R$. 
Write $c = (m-r+1)(n-r+1)$ the codimension of $I_X$. Let $f_1, \dots, f_c$ be general linear combinations of the generators $I_X$, for which we may assume that $f_i$ are forms of degree $r$ and $\codim (f_1,\dots, f_c) = c$. 
Set $I_V = (f_1, \dots, f_c)$ and $I_Y = I_V : _S I_X$. We have the following statements.
\begin{enumerate}[$(a)$]
\item The graded canonical module of $S/I_X$, denoted by $\omega_{S/I_X}$, is generated in degree $(r-1)m$.

\item $\omega_{S/I_X}(mn-rc) \cong I_Y/ I_V$.

\item The ideal $I_Y/I_V$ is generated in degree $rc - m(n-r+1)$.  In particular, the ideal $I_Y$ can be generated by the elements of $I_V$ of degree $r$ and elements of degree $$rc - m(n-r+1) = (n-r+1)(m-r)(r-1).$$
\end{enumerate}
\end{proposition}

\begin{proof}
For item (a), see \cite[Bottom of p.5]{BH92}, and item (b) follows from \cite[Prop.\ 5.2.6]{Mig98} with $rc$ and $mn-1$ for $t$ and $n$ in the proposition, respectively. 

Combining $(a)$ and $(b)$, we conclude that the ideal $I_Y/I_V$ is generated in degree 
$$(r-1)m -(mn-rc) = rc -m (n-r + 1).$$ 
Notice that with $c = (m-r+1)(n-r+1)$, we obtain 
\begin{align*}
rc -m (n-r + 1) &= r(m-r+1)(n-r+1) - m(n-r+1)\\
&= (n-r+1)(rm-r^2+r-m) \\
&= (n-r+1)(m-r)(r-1).\qedhere
\end{align*}
\end{proof}

\subsection{Summary of the proof of \Cref{mainthm1}}
\Cref{mainthm1} holds trivially when $\lct(A,X) = \codim_A X$.
Indeed, one has
\begin{equation*}
\lct(A,X) = \codim_A X = \codim_{A'} Y \geq \lct(A',Y).
\end{equation*} 
%where the last inequality holds by \cite[Lem. 2.5]{Niu14} or \cite[Ex. 9.2.14 and Prop. 9.2.32(a)]{Lar04}.
 For the general situation, our approach is to relate the order of $X$ calculated on the exceptional divisors of the log resolution of $X$ in \cite{Joh03} to the order of the generic link $Y$. 
We achieve this by extending such a log resolution $\pi \colon \tA \to A$ of $(A,X)$ to 
a log resolution $\pi' \colon \tA' \to A'$ of $(A', X')$ and compare $\ord_{E_i} (X)$ and $\ord_{E_i'} (Y)$, where $E_i'$ is the pull-back of the divisor $E_i$ to $\tA'$. 
Even though the extension $(A',X')$ of the log resolution of $(A,X)$ in \Cref{thm:res} does not extend to a log resolution of $(A',Y)$, one can find a log resolution of $(A',Y)$ extending the log resolution of $(A',X')$. 
In this log resolution, the strict transform of $E_i'$ appears in the set of exceptional divisors. 
By our definition of the log canonical threshold \Cref{def:lct}, one has $\lct (A',Y) \le  (k_i + 1)/\ord_{E_i'} (Y)$. 
In \Cref{thmAxAy}, we show that unless $\lct(A,X) = \codim (X,A)$, the value $a_{i,X}$ determining $\lct(A,X)$ satisfies $a_{i,X} =\ord_{E_i} (X ) =  \ord_{E_i'} (Y)$ which allows us to conclude equality of log canonical thresholds.

\section{Extension of the log resolution of $(A,X)$ and comparison of orders}
\label{genDegree}
We assume the notation and setting of \Cref{sec:setting}. 
In this section, we provide a closed form of $\ord_{E_i'} (Y)$ in terms of $m,n,r$ and $i$, where $E_i'$ is the pull-back of the divisor $E_i$ to $A' := A \times \Spec \CC[T]$. 
In the sequel, we will use the notation $-' := - \times_{\Spec \CC} \Spec \CC[T]$ to denote the extension of a subscheme of $A$ to $A'$.
%First, we extend the log resolution of the pair $(A, X)$ to $(A',X')$.
Since $A' \to A$ is an affine extension, the log resolution of $(A,X)$ extends naturally to $A'$, and $\ord_{E_i} (X) = \ord_{E_i'} (X')$.

\begin{prop}\label{thmAxAy}
We assume the notation and setting of \Cref{thm:res}. 
Let $Y$ be a generic link of $X$ in $A'$ by $V$.
Let $m_i = m-i+1, n_i = n-i+1, r_i =r - i +1$ for $i = 1,\dots, r$. Then for $i  \in \{ 1, \dots, r\}$, 
we have
\begin{equation*}
\ord_{E_i'} (Y) = \min \{ r_i, (n_i-r_i+1)(m_i-r_i)(r_i - 1) \}.
\end{equation*}
\end{prop}
\begin{proof}
Fix $i \in \{ 1, \dots, r\}$. 
Let $Z$ denote the strict transform of $V(I_i(M))$ in $A_{i-1}'$.
We have 
\begin{equation*}
\pi_{i}'^{-1} (Z) \subset  A_i' = \Bl_{Z} (A_{i-1}') \stackrel{\pi_{i}'}{\to} A_{i-1}'.
\end{equation*}
Thus, $E_i' = \pi_{i}'^{-1} (Z)$, $Z$ is smooth  in $A_{i-1}'$, and locally defined by the entries of a  $m_i \times n_i$ matrix of indeterminates. 
Furthermore, the $r_i$ by $r_i$ minors of this matrix of indeterminates locally define $X_{i-1}'$ in $A_{i-1}'$, see \Cref{rmk:Joh03}.
%These properties are preserved when we extend the log resolution of $X$ in \Cref{thm:res} to $A'$. \\
Note that $\ord_{E_i'} (Y) = \ord_{E_i'} ({Y_{i-1}}) = \sup \{ q \in \mathbf{N} \cup \{0\} \mid \mathcal{I}_{Y_{i-1}} \subset (\mathcal{I}_Z)^q \}$, where $(\mathcal{I}_Z)^0 := \mathcal{O}_{A_{i-1}}$.  
%We can check such an inclusion of ideal sheaves locally. 
Thus, we need to show that $\mathcal{I}_{Y_{i-1}} \subset (\mathcal{I}_Z)^q \setminus (\mathcal{I}_Z)^{q+1}$, where $q := \min \{ r_i, (n_i-r_i+1)(m_i-r_i)(r_i - 1) \}$, and this can be checked locally. 
To this end, we choose an affine cover of $A_{i-1}$ as in \cite[Sec.\ 4.2]{Joh03} and extend it to $A_{i-1}'$.

Fix an affine chart, and let $R_{i-1}$ denote the coordinate ring of the affine chart of $A_{i-1}$ and $S_{i-1} = R_{i-1}[T]$ denote the coordinate ring of the affine chart of $A_{i-1}'$.
Let $I_{X_{i-1}'}, I_{V_{i-1}},I_{Y_{i-1}}$, and $I_{Z}$ denote the corresponding ideals of $X_{i-1}', V_{i-1}, Y_{i-1}$ and $Z$ in $S_{i-1}$. 
Further, let $I_{X_{i-1}}$ denote the corresponding ideal of $X_{i-1}$ in $R_{i-1}$. 

% In $R_{i-1}$, $I_{X_{i-1}}$ can be generated by the $r_i$ by $r_i$ minors of a matrix of indeterminates. 
% Let $M'$ denote this matrix of indeterminates, so that $I_{X_{i-1}} = I_{r_i}(M')$ and $I_Z = I_1(M') S_{i-1}$. 
% Furthermore, the polynomial ring $R_{i-1}$ can be written as $R_{i-1}^\sim[M']$ for some polynomial subring $R_{i-1}^\sim$ of $R_{i-1}$. 
% Let $I_W$ be the generic link of $I_{X_{i-1}}$ with respect to the generating set consisting of the $r_i$ by $r_i$ minors of $M'$.

In $R_{i-1}$, $I_{X_{i-1}}$ can be generated by the $r_i$ by $r_i$ minors of a matrix of indeterminates. 
Let $M'$ denote this matrix of indeterminates, so that $I_{X_{i-1}} = I_{r_i}(M')$ and $I_Z = I_1(M') S_{i-1}$. 
Furthermore, the polynomial ring $R_{i-1}$ can be written as $\widetilde{R}_{i-1}[M']$ for some polynomial subring $\widetilde{R}_{i-1}$ of $R_{i-1}$. 
Let $I_W$ be the generic link of $I_{X_{i-1}}$ with respect to the generating set consisting of the $r_i$ by $r_i$ minors of $M'$.

First, we claim that $I_{W} \in I_Z^q \setminus I_Z^{q+1}$. 

% Since $S_{i-1} = (R_{i-1}^\sim[M'])[T] = (R_{i-1}^\sim[T])[M']$, 
% by writing $S_{i-1}^\sim = R_{i-1}^\sim[T]$, we can and will view $S_{i-1}$ as a polynomial ring over $S_{i-1}^\sim$. 
% Recall that $I_Z = I_1(M')$ and $I_{X_{i-1}} = I_{r_i} (M')$. 
% By \Cref{genDegreeLem}(c), the ideal $I_{W}$ is generated by elements of degrees $r_i$ and $(n_i-r_i+1)(m_i-r_i)(r_i - 1)$. 
% This proves the claim of $I_{W} \in I_Z^q \setminus I_Z^{q+1}$.

Since $S_{i-1} = (\widetilde{R}_{i-1}[M'])[T] = (\widetilde{R}_{i-1}[T])[M']$, 
by writing $\widetilde{S}_{i-1} = \widetilde{R}_{i-1}[T]$, we can and will view $S_{i-1}$ as a polynomial ring over $\widetilde{S}_{i-1}$. 
Recall that $I_Z = I_1(M')$ and $I_{X_{i-1}} = I_{r_i} (M')$. 
By \Cref{genDegreeLem}(c), the ideal $I_{W}$ is generated by elements of degrees $r_i$ and $(n_i-r_i+1)(m_i-r_i)(r_i - 1)$. 
This proves the claim of $I_{W} \in I_Z^q \setminus I_Z^{q+1}$.

%%\begin{equation*}
%%\ord_{E_i'}  (Y_{i-1})= \min \{ r_i, (n_i-r_i+1)(m_i-r_i)(r_i - 1) \}.
%%\end{equation*}

The above generating degree is uniform in each affine chart. 
So, to finish the proof, it suffices to show that $I_{Y_{i-1}}$ is another generic link of $I_{X_{i-1}}$ in $S_{i-1}$. 
Once we have shown this, then by \Cref{lem:algOrd}, the order of the ideals $I_{Y_{i-1}}$ and $I_W$ with respect to $I_Z$ are equal, and this completes the proof. 

Now, we claim that $I_{Y_{i-1}}$ is a generic link of $I_{X_{i-1}}$. 
Let $\Delta_1, \dots, \Delta_{\mu}$ denote the $r$ by $r$ minors of the matrix of indeterminates $M$, and let $f_1, \dots, f_c$ denote the equations defining $V$. Thus, $f_j$ are of the form
\begin{equation*}
f_j = t_{j1} \Delta_1 + \cdots + t_{j \mu} \Delta_\mu.
\end{equation*} 
For $1 \le \ell \le i-1$, let $y_\ell$ denote the local equation of the strict transforms of $V(I_i(M))$ to $A'_{i-1}$ in the chosen affine chart.
Then since $\ord_{E_\ell} (X) = \ord_{E_{\ell}'}  (V) = r_{\ell}$, 
${I}_{X_{i-1}'}$ and ${I}_{V_{i-1}'}$ are generated by the sets of equations
\begin{equation*}
\{ \Delta_1/y, \dots, \Delta_{\mu}/y \} \quad\text{and}\quad \{ f_1/y, \dots, f_c/y \},
\end{equation*} 
respectively, where $y := \prod_{\ell = 0}^{i-1} y_\ell^{r_\ell}$. 
In particular, one has
\begin{equation*}
\frac{f_j}y = t_{j1} \frac{\Delta_1}y + \cdots + t_{j \mu} \frac{\Delta_\mu}y,
\end{equation*} 
and by \cite[Claim 3.1.2 part (3)]{Niu14}, $f_1/y, \dots, f_c/y$ form a regular sequence in $S_{i-1} = R_{i-1}[T]$, and $I_{Y_{i-1}} = {I}_{V_{i-1}'} :_{S_{i-1}} {I}_{X_{i-1}'} =  {I}_{V_{i-1}'} :_{S_{i-1}} {I}_{X_{i-1}}$.
This proves the second claim and completes the proof.  
\end{proof}

\section{Proof of \Cref{mainthm1}}

The following lemma, whose proof is elementary and left to the reader, will be useful for the proof of \Cref{mainthm1}.  

\begin{lemma}\label{lem:qibound}
For fixed integers $1 \leq r \leq n \leq m$, $(n-r+1)(m-r)(r-1) < r$ if and only if
\begin{enumerate}[$(a)$]
\item $r = 1$, 
\item $m = r$, or
\item $n = r$ and $m = r + 1$.
\end{enumerate}
\end{lemma}
In particular, for $i = 1,\dots, r-1$, set $m_i = m - i + 1, n_i = n-i+1, r_i = r-i+1$, then 
\begin{equation*}
(n-r+1)(m-r)(r-1) < r \iff (n_i-r_i+1)(m_i-r)(r_i-1) < r_i.
\end{equation*}

Note that case (b) is the determinant of a square matrix, and case (c) is the maximal minors of $m$ by $m-1$ matrices. 
In case (c), the generic link is also determinantal of size $m+1$ by $m$. 

\begin{prop}[{cf. \cite[Theorem 6.4]{Joh03}}]\label{proplctC}
Let $X = \Spec \CC[M]/I_r(M)$ be the variety defined by the $r \times r$-minors of an $m \times n$-matrix $M$ of indeterminates, and let $Y$ be the generic link of $X$. 
If $(a)~r = 1, (b)~m=r$, or $(c)~n = r$ and $m = r+1$, i.e., the three cases in \Cref{lem:qibound}, then $\lct (A,X) = \codim_A X$. 
In particular, if $(n-r+1)(m-r)(r-1) < r$, then $\lct (A,X) = \lct (A',Y)$.
\end{prop}
\begin{proof}
Apply the formula, in Equation~\eqref{detLCTformula} in \Cref{thm:res} with \Cref{lem:qibound}.
\end{proof}

Now, we are ready to proof our main theorem.
\begin{proof}[Proof of \Cref{mainthm1}:] 
We assume the notation and setting of \Cref{thm:res}. 

If $(n-r+1)(m-r)(r-1) < r$, then we are done by \Cref{proplctC}. 
Without loss of generality, we may assume that $(n-r+1)(m-r)(r-1) \ge r$. 

Suppose $a_{i,X}$ with $1 \leq i \leq r$ is the order computing the log canonical threshold of $X$. 
That is $\lct(A,X) = \frac{k_i + 1}{a_{i,X}}$. 
Once we have shown that 
\begin{align*}
\lct (A,X) = \codim_A X \qquad &\text{if}~ i = r ~\text{and} \\
\ord_{E_i} (X) \le \ord_{E_i'} (Y) \qquad &\text{otherwise}, 
\end{align*}
the equality $\lct (A,X) = \lct (A',Y)$ follows by \Cref{eqEq}. 
If $i = r$, then $\lct(A,X) = \codim_A X$. Hence $\lct(A',Y) = \lct (A,X)$ by \Cref{eqEq}. 
Without loss of generality assume $1 \leq i < r$. 
In this range of $i$, by \Cref{lem:qibound}, we have 
\begin{equation*} 
(n-r+1)(m-r)(r-1) \ge r \iff (n_i-r+1)(m_i-r)(r_i-1) \ge r_i, 
\end{equation*} 
where $m_i = m-i+1, n_i =n-i+1$, and $r_i = r-i+1$.
Therefore, by \Cref{thmAxAy}, we conclude that 
$\ord_{E_i'}(Y) =  r_i.$
\
Since $\ord_{E_i'} (X) = r_i$, this completes the proof. 
\end{proof}

\begin{cor}
In the setting of \Cref{mainthm1}, if $(n-r+1)(m-r)(r-1) \ge r$, then $\ord_{E_r'} (Y) =0$ and $\ord_{E_i} (X) = \ord_{E_i'} (Y)$ for $i = 1,\dots, r-1$. 
\end{cor}
\begin{proof}
By the proof of \Cref{mainthm1}, it suffices to show the assertion $\ord_{E_r'} (Y) = 0$. 
Notice that $X_r'$ is smooth , and it is the center of the blow-up $\pi_r$. 
Since $Y_r$ and $X_r'$ are geometrically linked, $X_r' \not\subset Y_r$. 
This shows that $I_Y \cO_{X_r'} = \cO_{X_r'}$, hence $\ord_{E_r'} (Y) = 0$. 
\end{proof}

\begin{remark}[{cf.\ \Cref{thmAxAy}}]
The reader should be warned not to jump to the incorrect conclusion that one has 
$\ord_{E_i}(X) = \ord_{E_i'} (Y)$ in general. 
Let 
\begin{equation*}
M = \begin{pmatrix} x_{11} & x_{12} \\
x_{21} & x_{22} \\
x_{31} & x_{32} \\ \end{pmatrix}, \qquad
T = \begin{pmatrix}
t_{11} & t_{12} & t_{13} \\ 
t_{21} & t_{22} & t_{23}
\end{pmatrix},
\end{equation*}
and $\Delta_i$ denote the signed minor of $M$ after deleting the $i$th row for $i = 1,2,3$. 
Then $I_V$ is generated by two elements 
\begin{equation*}
v_1 = t_{11}\Delta_1 + t_{12}\Delta_2 + t_{13} \Delta_3, \quad v_2 = t_{21} \Delta_1 +  t_{22}\Delta_2 + t_{23} \Delta_3,
\end{equation*}
and $I_Y = I_V : I_X$ is generated by the $3 \times 3$ minors of the matrix
\begin{equation*}
\begin{pmatrix} x_{11} & x_{21} & x_{31} \\
 x_{12} & x_{22} & x_{32} \\ 
t_{11} & t_{12} & t_{13} \\ 
t_{21} & t_{22} & t_{23}
\end{pmatrix}.
\end{equation*}
Since the ideal generated by the variables in $x_{i,j}$, say $\m$, is the center of the blow up
\begin{equation*}
\pi_1 : A_1 := \Bl_{V(I_1(M))} (A) \to A := \Spec \CC[M,T],
\end{equation*}
 and $I_Y \in \m \setminus \m^2$.
Thus $\ord_{E_i} (X) = 2$, but $\ord_{E_i'} (Y) = 1$. 
\end{remark}
\

\begin{remarkQuestion}~We conclude with a few quesitons.
\begin{enumerate}[$(a)$]
\item The generic link we use in the paper is called a 1st generic link. One can construct higher generic links iteratively by taking a generic link of a generic link, and etc. These generic links were used heavily in the study of the structure of licci ideals in \cite{HU87}. 

A natural question to ask is for a variety $X$ and its generic link $Y$, if the log canonical threshold of $X$ and $Y$ are equal, then are the log canonical threshold of higher generic links equal to that of $X$ and $Y$?

A forth-coming paper of the first author will address this question. 

\item 
In general, it is difficult to determine the structure of a generic link. 
One class for which we have an explicit description of a generic link is the class of complete intersections. 
Suppose $X$ is a complete intersection subscheme of a smooth quasi-projective scheme $A$, and $Y$ is its generic link in $A'$.
In general, inequality $\lct (A',Y) \ge \lct (A,X)$ can be strict, see \Cref{xmp:codim2}.
Even though we have explicit descriptions of generic links, the equations defining them become far more complicated, making the study of its log canonical threshold challenging. 
%The difficult part is that even if $X$ is a complete intersection, the structure of its generic link gets complicated and at present it does not seem clear  how to determine the structure of such a generic link in general, let alone compute its log canonical threshold. 
% The generic link is known but its lct is not. 
Thus, it will be interesting to find a subclass of the class of complete intersections where inequality in \Cref{eqEq} is either equal or a strict inequality. 
\end{enumerate}
\end{remarkQuestion}

\end{document}